\documentclass[11pt]{amsart}
\usepackage{cases}

\theoremstyle{plain}
\newtheorem{theorem}                {Theorem}      [section]
\newtheorem*{theorem*}                {Theorem}
\newtheorem{proposition}  [theorem]  {Proposition}
\newtheorem{corollary}    [theorem]  {Corollary}
\newtheorem{lemma}        [theorem]  {Lemma}

\theoremstyle{definition}

\newtheorem{remark}       [theorem]  {Remark}
\newtheorem{definition}   [theorem]  {Definition}

\DeclareMathOperator{\trace}{trace} 

\DeclareMathOperator{\Div}{div}

\DeclareMathOperator{\cst}{constant}

\DeclareMathOperator{\sol}{Sol_3}

\numberwithin{equation}{section}

\begin{document}

\title[Biconservative surfaces]{Biconservative surfaces in $\mathbb{S}^2\times\mathbb{R}$ and $\mathbb{H}^2\times\mathbb{R}$}

\author{Dorel~Fetcu}

\thanks{}

\address{Department of Mathematics and Informatics\\
Gh. Asachi Technical University of Iasi\\
Bd. Carol I, 11A \\
700506 Iasi, Romania} \email{dorel.fetcu@academic.tuiasi.ro}



\begin{abstract} We find the explicit local equations of biconservative surfaces with non-constant mean curvature in $\mathbb{S}^2\times\mathbb{R}$ and $\mathbb{H}^2\times\mathbb{R}$, when the gradient of the mean curvature function is a principal direction.
\end{abstract}

\maketitle

\section{Introduction}

Biharmonic maps between two Riemannian manifolds are critical points of the $L^{2}$-norm of the tension field. This variational problem was suggested in the mid 1960s by J.~Eells and J.~H.~Sampson in their seminal paper \cite{ES} and it was first studied by~G.-Y.~Jiang \cite{Jiang}.

In approximately the same period, B.-Y.~Chen \cite{43} defined biharmonic submanifolds in Euclidean spaces $\mathbb{R}^n$ as isometric immersions with harmonic mean curvature vector field. We note that the two definitions of biharmonicity agree in $\mathbb{R}^n$. Since biharmonic surfaces in $\mathbb{R}^{3}$ proved to be minimal (see \cite{Ishikawa, J2}), this result led to Chen's Conjecture that all biharmonic submanifolds of Euclidean spaces are minimal \cite{43}. Although a great number of articles gave positive partial answers to support it, the conjecture is still open.   

However, there are many interesting examples of proper-biharmonic submanifolds, i.e., biharmonic non-minimal, in spaces like Euclidean spheres, complex projective spaces, and other ambient manifolds with convenient curvature properties (for detailed accounts see \cite{FO, Chen-Ou}). 

The notion of a biconservative submanifold has developed from the theory of biharmonic submanifolds by only requiring the vanishing of the tangent part of the bitension field. Although a rather new one, this subject is already well established. Restricting only to biconservative surfaces, we could mention some recent and (a few) older articles like \cite{AK, CMOP, Hasanis, LMO, MOR, MOP, SN, Nistor1, NO} as illustrative examples for studies on this topic. 

Biharmonic and biconservative surfaces in product spaces $\mathbb{S}^2\times\mathbb{R}$ and $\mathbb{H}^2\times\mathbb{R}$ were studied in \cite{MOP,OW}, with the additional hypothesis that these surfaces also have constant mean curvature (CMC surfaces). This situation proved to be quite rigid since it turned out that the only CMC biconservative surfaces are Hopf tubes over curves with constant geodesic curvature (see \cite{MOP}). Since $\mathbb{S}^2\times\mathbb{R}$ and $\mathbb{H}^2\times\mathbb{R}$ are two of the eight Thurston geometries, it is interesting to have a glimpse of the studies on non-CMC biconservative surfaces in the other six ambient spaces. Until now, this problem has been addressed only for surfaces of $\sol$ \cite{FSol} and, in a very extensive way, in space forms. The explicit local equations of non-CMC biconservative surfaces in space forms were obtained  in \cite{CMOP}, then complete such surfaces were constructed in \cite{SN} and their  uniqueness proved in \cite{NO}. The existence of compact non-CMC biconservative surfaces in the Euclidean sphere $\mathbb{S}^3$ was proved in \cite{MP}.   

In the case of non-CMC biconservative surfaces in space forms it turned out that the gradient of the mean curvature function always is a principal direction (see \cite{CMOP}). It seems then interesting to consider non-CMC biconservative surfaces with this property in product spaces $M^2(c)\times\mathbb{R}$, where $M^2(c)$ is either $\mathbb{S}^2$ or $\mathbb{H}^2$, as $c=1$ or $c=-1$. In our paper, we study the geometry of such surfaces and then find their explicit local equations. 

\textbf{Conventions.}
Throughout the paper surfaces are oriented, we use the following sign conventions
$$
R(X,Y)Z=\nabla_{X}\nabla_{Y}Z-\nabla_{Y}\nabla_{X}Z-\nabla_{[X,Y]}Z,\quad
\Delta=\trace\nabla^{2}=\trace(\nabla\nabla-\nabla_{\nabla}),
$$
and objects on $M(c)\times\mathbb{R}$ are indicated by $\overline{(\cdot)}$, while those on $\mathbb{R}^4$ or $\mathbb{L}^3\times\mathbb{R}$ are denoted by $\widetilde{(\cdot)}$.

\textbf{Acknowledgments.} The author would like to thank Cezar Oniciuc for useful discussions and suggestions, and also Max Planck Institute for Mathematics in Bonn for its hospitality and financial support.

\section{Preliminaries}

Biharmonic maps $\phi : M^{m} \to N^{n}$ between two Riemannian manifolds are critical points of the bienergy functional
$$
E_{2}:C^{\infty}(M,N)\to \mathbb{R}, \quad E_{2}(\phi)=\frac{1}{2}\int_{M} |\tau(\phi)|^{2} dv,
$$
where $\tau(\phi)= \trace  \nabla d \phi$ is the tension field of $\phi$. The Euler-Lagrange equation, also called the biharmonic equation, was derived by G.-Y.~Jiang \cite{Jiang}
\begin{eqnarray}\label{tau-2}
\tau_{2}(\phi)=\Delta \tau(\phi)-\trace \bar R(d\phi(\cdot),\tau(\phi))d \phi (\cdot)=0,
\end{eqnarray}
where $\tau_{2}(\phi)$ is the bitension field of $\phi$ and $\bar R$ is the curvature tensor field of $N$.

It is easy to see that any harmonic map is biharmonic and, therefore, we are interested in the non-harmonic biharmonic ones, called proper-biharmonic.

Next, if we consider a given map $\phi$ and let the domain metric vary, one obtains a functional 
$$
\mathcal{F}_{2}:\mathcal{G}\to \mathbb{R}, \quad \mathcal{F}_{2}(g)=E_{2}(\phi),
$$
on the set $\mathcal{G}$ of Riemannian metrics on $M$. Critical points of this functional are characterized by the vanishing of the stress-energy
tensor $S_{2}$ of the bienergy (see \cite{LMO}). This tensor, introduced in \cite{Jiang87}, is given by
\begin{eqnarray*}
S_{2}(X,Y)&=&\frac{1}{2}\vert \tau (\phi)\vert ^{2}\langle X,Y \rangle +\langle d \phi, \nabla \tau (\phi) \rangle \langle X, Y \rangle-\langle d\phi (X),\nabla_{Y} \tau (\phi)\rangle
\\
&\ & -\langle d\phi (Y),\nabla_{X} \tau (\phi)\rangle,
\end{eqnarray*}
and satisfies
$$
\Div S_{2}=-\langle \tau_{2}(\phi), d\phi\rangle.
$$

For isometric immersions, $(\Div S_{2})^{\sharp} =-\tau_{2}(\phi)^{\top}$, where $\tau_{2}(\phi)^{\top}$ is the
tangent part of the bitension field.

\begin{definition}
A submanifold $\phi:M^{m} \to N^{n}$ of a Riemannian manifold $N^{n}$ is called biconservative if $\Div S_{2}=0$.
\end{definition}

It is again easy to see that a submanifold is biconservative if and only if the tangent part of the bitension field vanishes.

Now, let us consider a surface $\Sigma^2$ in a Riemannian manifold $N^n$. The Gauss and the Weingarten equations of the surface
$$
\bar\nabla_XY=\nabla_XY+\sigma(X,Y)\quad\textnormal{and}\quad \bar\nabla_XV=-A_VX+\nabla^{\perp}_XV,
$$
hold for all vector fields $X$ and $Y$ tangent to the surface and $V$ normal to $\Sigma^2$, where $\nabla$ is the induced connection on $\Sigma^2$, $\sigma$ is the second fundamental form, $A$ the shape operator, and $\nabla^{\perp}$ is the connection in the normal bundle. If the ambient space $N$ is three-dimensional, the mean curvature vector field of $\Sigma^2$ is given by $H=f\eta$, where $f=(1/2)\trace A$ is the mean curvature function and $\eta$ is a unit vector field normal to $\Sigma^2$. 

If $f$ is constant, then $\Sigma^2$ is called a constant mean curvature (CMC) surface.

The Codazzi equation of $\Sigma^2$ is
\begin{equation}\label{eq:Codazzi}
\begin{array}{cl}
\langle \bar R(X,Y)Z,V\rangle=&\langle(\nabla^{\perp}_X\sigma)(Y,Z),V\rangle-\langle(\nabla^{\perp}_Y\sigma)(X,Z),V\rangle,
\end{array}
\end{equation}
for any vector fields $X$, $Y$, $Z$ tangent to $\Sigma^2$, and any normal vector field $V$, where 
$$
(\nabla^{\perp}_X\sigma)(Y,Z)=\nabla^{\perp}_X\sigma(Y,Z)-\sigma(\nabla_XY,Z)-\sigma(Y,\nabla_XZ).
$$

The Gauss equation of the surface is 
\begin{equation}\label{eq:Gauss}
\langle R(X,Y)Z,W\rangle=\langle\bar R(X,Y)Z,W\rangle+\langle AY,Z\rangle\langle AX,W\rangle-\langle AX,Z\rangle\langle AY,W\rangle,
\end{equation}
for all tangent vector fields $X$, $Y$, $Z$, and $W$, where $R$ is the curvature tensor of $\Sigma^2$.

The following theorem, that splits the biharmonic equation into its normal and tangent parts, was proved in the more general case of hypersurfaces in any Riemannian manifold in \cite{Ou} (see also \cite{LMO} for the case of arbitrary codimension).

\begin{theorem}[\cite{Ou}]\label{decomposition}
A surface $\Sigma^2$ in a Riemannian manifold $N^3$ is biharmonic if and only if
$$
\begin{cases}
\Delta f-f|A|^2-f\langle\trace(\bar R(\cdot,\eta)\cdot),\eta\rangle=0\\
A(\nabla f)+f\nabla f+f\trace(\bar R(\cdot,\eta)\cdot)^{\top}=0.
\end{cases}
$$ 
\end{theorem}

\begin{corollary}\label{c:bicons}
A surface $\Sigma^2$ in a Riemannian manifold $N^3$ is biconservative if and only if
\begin{equation}\label{eq:bicons}
A(\nabla f)+f\nabla f+f\trace(\bar R(\cdot,\eta)\cdot)^{\top}=0.
\end{equation}
\end{corollary}

We end this section by recalling that the expression of the curvature  tensor $\bar R$ of a product space $M^2(c)\times\mathbb{R}$, where $M^2(c)$ is a space form, is given by
\begin{equation}\label{eq:barR}
\begin{array}{ll}
\bar R(X,Y)Z=&c\{\langle Y, Z\rangle X-\langle X, Z\rangle Y-\langle Y,\xi\rangle\langle Z,\xi\rangle X+\langle X,\xi\rangle\langle Z,\xi\rangle Y\\ \\&+\langle X,Z\rangle\langle Y,\xi\rangle\xi-\langle Y,Z\rangle\langle X,\xi\rangle\xi\},
\end{array}
\end{equation}
where $\xi$ is the unit vector tangent to $\mathbb{R}$.

\section{Biconservative surfaces in $M^2(c)\times\mathbb{R}$}

The only known examples of biconservative surfaces in product spaces  were, until now, Hopf tubes over curves with constant geodesic curvature. In \cite{MOP} the authors proved that such tubes are the only CMC biconservative surfaces in these spaces. In the following we will consider the case of non-CMC biconservative surfaces in $M^2(c)\times\mathbb{R}$, with the additional property that $\nabla f$ is a principal direction. 

Let $\Sigma^2$ be a biconservative surface in $M^2(c)\times\mathbb{R}$ such that $\nabla f\neq 0$ at a point $p\in\Sigma^2$. There exists a neighborhood $V$ of $p$ with $\nabla f\neq 0$ on $V$. Since $f$ cannot vanish on $V$, there exists an open subset $U$ of $V$ such that $f\neq 0$ at each point of $U$. Moreover, we can assume that $f>0$ on $U$. Next, let us define an orthonormal frame field $\{X_1,X_2\}$ on $U$ as
$$
X_1=\frac{\nabla f}{|\nabla f|},\quad X_2\perp X_1,\quad |X_2|=1.
$$
One can now see that
$$
\nabla f=X_1(f)X_1+X_2(f)X_2
$$
leads to
\begin{equation}\label{eq:nablaf}
X_1(f)=|\nabla f|>0\quad\textnormal{and}\quad X_2(f)=0.
\end{equation}

Henceforth, throughout the paper, we will assume that $\nabla f$ is a principal direction. Then, from Corollary \ref{c:bicons} and formula \eqref{eq:barR}, the biconservative condition is equivalent to
\begin{equation}\label{eq:l1}
AX_1=-\left(f+\frac{cf}{|\nabla f|}\sin\theta\cos\theta\right)X_1=\lambda_1X_1.
\end{equation}
It follows that also $X_2$ is a principal direction and 
\begin{equation}\label{eq:l2}
AX_2=\left(3f+\frac{cf}{|\nabla f|}\sin\theta\cos\theta\right)X_2=\lambda_2X_2.
\end{equation}
Moreover, we also have that $X_2\perp\xi$ and we can write
\begin{equation}\label{eq:xi}
\xi=\cos\theta X_1+\sin\theta\eta
\end{equation}
where $\theta$ is a non-constant real valued function on $U$ and $H=f\eta$, with $\eta$ a unit vector field normal to the surface. We note that $\theta$ cannot be a constant since for biconservative surfaces in $M^2(c)\times\mathbb{R}$ this is equivalent to the fact that the surface is CMC (see \cite{MOP}).

\begin{lemma}\label{lemma1}
On the set $U$ the following identities hold:
\begin{enumerate}

\item $X_1(\theta)=-\lambda_1$;

\item $X_2(\theta)=0$;

\item $\cos\theta\neq 0$ and $\sin\theta\neq 0$ at each point in $U$;

\item $\nabla_{X_1}X_1=\nabla_{X_1}X_2=0$;

\item $\nabla_{X_2}X_1=\lambda_2\tan\theta X_2$;

\item $\nabla_{X_2}X_2=-\lambda_2\tan\theta X_1$;

\item $X_1(\lambda_2)+\lambda_2(\lambda_2-\lambda_1)\tan\theta+c\sin\theta\cos\theta=0$;

\item $X_2(\lambda_1)=X_2(\lambda_2)=0$.

\end{enumerate}
\end{lemma}

\begin{proof} The vector field $\xi$ is parallel and then, from \eqref{eq:xi}, we have
\begin{eqnarray*}\label{eq:1}
\bar\nabla_{X_1}\xi&=&-X_1(\theta)\sin\theta X_1+\cos\theta(\nabla_{X_1}X_1+\sigma(X_1,X_1))+X_1(\theta)\cos\theta\eta-\sin\theta AX_1\\&=&0
\end{eqnarray*}
and
\begin{eqnarray*}\label{eq:2}
\bar\nabla_{X_2}\xi&=&-X_2(\theta)\sin\theta X_1+\cos\theta(\nabla_{X_2}X_1+\sigma(X_1,X_2))+X_2(\theta)\cos\theta\eta-\sin\theta AX_2\\&=&0.
\end{eqnarray*}
Since $X_1$ and $X_2$ are principal directions, the first two items of the lemma readily follow. Further, by taking the inner product with $X_2$, we have
$$
\cos\theta\langle\nabla_{X_1}X_1,X_2\rangle=0\quad\textnormal{and}\quad
\cos\theta\langle\nabla_{X_2}X_1,X_2\rangle=\lambda_2\sin\theta.
$$

Let us assume now that $\cos\theta=0$ at a point $q\in U$. Then $\lambda_2=0$ at $q$, but this cannot happen since $\lambda_2=3f$ at such a point, as shown by \eqref{eq:l2}. This means that $\cos\theta\neq 0$ throughout $U$ and then $\nabla_{X_1}X_1=0$, $\nabla_{X_2}X_1=\lambda_2\tan\theta X_2$, and $X_2(\lambda_1)=0$. Since $X_2(f)=0$, we also have $X_2(\lambda_2)=0$.

From Codazzi equation \eqref{eq:Codazzi}, it follows that
$$
\langle(\nabla_{X_1}A)X_2-(\nabla_{X_2}A)X_1,Z\rangle=\langle\bar R(X_1,X_2)Z,\eta\rangle,
$$
for all tangent vectors $Z$. Using \eqref{eq:barR}, one easily obtains
$$
X_1(\lambda_2)+\lambda_2(\lambda_2-\lambda_1)\tan\theta+c\sin\theta\cos\theta=0.
$$

Finally, if $\sin\theta=0$ at a point $q\in U$, then, at this point, $\nabla_{X_2}X_1=0$, which implies $X_1(\lambda_2)=0$ at $q$. From the biconservative equation \eqref{eq:l2}, we know that, on $U$, we have
$$
X_1(\lambda_2)=3X_1(f)+X_1\left(\frac{f}{|\nabla f|}\right)\sin\theta\cos\theta+\frac{f}{|\nabla f|}X_1(\theta)\cos(2\theta).
$$
Therefore, since $X_1(\theta)=-\lambda_1$, at $q$ one obtains
$$
3|\nabla f|-\lambda_1\frac{f}{|\nabla f|}=0,
$$
i.e., $\lambda_1=3|\nabla f|^2/f>0$. On the other hand, still at $q$, equation \eqref{eq:l1} gives $\lambda_1=-f<0$, which is a contradiction and we conclude. 
\end{proof}

\begin{remark} The Gauss equation \eqref{eq:Gauss} of our surface $\Sigma^2$ is
$$
K=\lambda_1\lambda_2+c\sin^2\theta,
$$
where $K$ is the Gaussian curvature of $\Sigma^2$. Computing $K$ by using Lemma \ref{lemma1} we can see that the Gauss equation does not lead to any new information.
\end{remark}

\begin{proposition}\label{prop1} The eigenfunction $\lambda_2$ can be written as $\lambda_2=a\cos\theta$, where $a$ is a real valued function on $\Sigma^2$ such that
$$
X_1(\theta)=a\cos\theta-2f,\quad X_1(a)+(a^2+c)\sin\theta=0,\quad\textnormal{and}\quad X_2(a)=0.
$$
\end{proposition}

\begin{proof} As we have seen in Lemma \ref{lemma1}, $\cos\theta\neq 0$ at each point of $U$. Then we can write $\lambda_2=a\cos\theta$, with $a:\Sigma^2\to\mathbb{R}$. Since $\lambda_1=-X_1(\theta)$ and $\lambda_1+\lambda_2=2f$, it follows that $X_1(\theta)=a\cos\theta-2f$. Moreover, $X_2(\lambda_2)=0$, $X_2(\theta)=0$, and $X_2(f)=0$ imply that also $X_2(a)=0$.

Next, again from Lemma \ref{lemma1}, we know that
$$
X_1(\lambda_2)+\lambda_2(\lambda_2-\lambda_1)\tan\theta+c\sin\theta\cos\theta=0,
$$
which, after a straightforward computation, using $X_1(\theta)=a\cos\theta-2f$ and again the fact that $\cos\theta\neq 0$ everywhere, gives
$$
X_1(a)+(a^2+c)\sin\theta=0,
$$
and this concludes the proof.
\end{proof}

\subsection{Biconservative surfaces in $\mathbb{S}^2\times\mathbb{R}$} In the following, the space form $M^2(c)$ will be the Euclidean sphere $\mathbb{S}^2$, i.e., $c=1$. We look at $\mathbb{S}^2\times\mathbb{R}$ as a submanifold in $\mathbb{R}^3\times\mathbb{R}=\mathbb{R}^4$. Our first main result gives the explicit local equation of a non-CMC biconservative surface in $\mathbb{S}^2\times\mathbb{R}$, with $\nabla f$ a principal direction, viewed as a surface in $\mathbb{R}^4$.

\begin{theorem}\label{thm:local_sxr} Let $\Sigma^2$ be a biconservative surface in $\mathbb{S}^2\times\mathbb{R}$ with $f>0$ and $\nabla f\neq 0$ at each point on the surface, such that $\nabla f$ is a principal direction. Then, locally, $\Sigma^2\subset\mathbb{R}^4$ can be parametrized as
\begin{equation}\label{eq:local_sxr}
\Phi(u,v)=\left(\frac{\cos v}{\sqrt{a^2+1}}C_1+\frac{\sin v}{\sqrt{a^2+1}}C_2+\frac{a}{\sqrt{a^2+1}}(C_1\times C_2),\int_{u_0}^u\cos\theta(t)dt\right),
\end{equation}
where $C_1,C_2\in\mathbb{R}^3$ are two constant orthonormal vectors, and $a=a(u)$ and $\theta=\theta(u)$ are two real valued functions such that 
$$
\sin\theta\neq 0,\quad \cos\theta\neq 0,\quad a'+(a^2+1)\sin\theta=0,
$$ 
and
\begin{equation}\label{eq:bicons_sxr}
\left(\theta'-a\cos\theta\right)\left(\theta'-a\cos\theta\right)'+(a^2+1)\left(\frac{(\theta')^2-a^2\cos^2\theta}{a^2+1}\right)'=0.
\end{equation}
\end{theorem}

\begin{proof} Let $\Sigma^2$ be a non-CMC biconservative surface in $\mathbb{S}^2\times\mathbb{R}$. With the same notations as in the previous section, from Lemma \ref{lemma1} and Proposition \ref{prop1}, we have $[X_1,X_2]=-\nabla_{X_2}X_1=-\lambda_2\tan\theta X_2=-a\sin\theta X_2$. This shows that we can choose local coordinates $\{u,v\}$ on the surface such that $\partial/\partial u=X_1$ and $\partial/\partial v=\beta(u,v)X_2$, where $\beta$ is a real valued function given by
\begin{equation}\label{eq:beta}
\frac{\partial\beta}{\partial u}=a\sin\theta\beta.
\end{equation}
Since $X_2(a)=X_2(\theta)=0$, it follows that $a=a(u)$ and $\theta=\theta(u)$. Moreover, again from Lemma \ref{lemma1} and Proposition \ref{prop1}, we have $\cos\theta\neq 0$ and $\sin\theta\neq 0$ at all points, $a'(u)=X_1(a)=-(a^2+1)\sin\theta$, and, as it can be easily verified, the biconservative equation \eqref{eq:l1} is just equation \eqref{eq:bicons_sxr}.

Next, from equation \eqref{eq:beta} one obtains $\beta(u,v)=\tau(v)/\sqrt{a^2+1}$, where $\tau=\tau(v)$ is a real valued function.

In the following we will derive the local expression of $\Phi(u,v)$ in terms of the functions $a$ and $\theta$.

The local equation of the surface $\Sigma^2$ in $\mathbb{R}^3\times\mathbb{R}=\mathbb{R}^4$ is 
$$
\Phi=\Phi(u,v)=(\Phi_1(u,v),\Phi_2(u,v),\Phi_3(u,v),\Phi_4(u,v)),
$$
with 
$$
\frac{\partial}{\partial u}=\Phi_u\quad\textnormal{and}\quad\frac{\partial}{\partial v}=\Phi_v.
$$

First, we have $(\Phi_4)_u=\langle\Phi_u,\xi\rangle=\cos\theta$ and $(\Phi_4)_v=\langle\Phi_v,\xi\rangle=0$. It follows that
\begin{equation}\label{eq:phi4.1}
\Phi_4(u,v)=\Phi_4(u)=\int_{u_0}^u\cos\theta(t)dt.
\end{equation}

Now, let us denote by $\widetilde\xi$ a unit vector field in $\mathbb{R}^4$ normal to $\mathbb{S}^2\times\mathbb{R}$. Along $\Sigma^2$ this vector field is given by $\widetilde\xi=(\Phi_1,\Phi_2,\Phi_3,0)$. The unit vector field $\eta$, normal to the surface in $\mathbb{S}^2\times\mathbb{R}$, viewed as a vector field in $\mathbb{R}^4$, is given by $\eta=(\eta_1,\eta_2,\eta_3,\sin\theta)$. 

Then, in $\mathbb{R}^4$, for any vector field $X$ tangent to $\Sigma^2$, we have, from the Weingarten equation of the surface in $\mathbb{R}^4$,
$$
\widetilde\nabla_X\widetilde\xi=-\widetilde A_{\widetilde\xi}X+\widetilde\nabla^{\perp}_X\widetilde\xi
$$
and 
$$
\widetilde\nabla^{\perp}_X\widetilde\xi=\langle\widetilde\nabla_X\widetilde\xi,\eta\rangle\eta=-\sin\theta\langle X,\xi\rangle\eta.
$$ 
It follows that 
\begin{equation}\label{eq:atilde}
\left\langle\widetilde A_{\widetilde\xi}\frac{\partial}{\partial u},\frac{\partial}{\partial u}\right\rangle=-\sin^2\theta,\quad \left\langle\widetilde A_{\widetilde\xi}\frac{\partial}{\partial v},\frac{\partial}{\partial v}\right\rangle=-\beta^2,\quad \left\langle\widetilde A_{\widetilde\xi}\frac{\partial}{\partial u},\frac{\partial}{\partial v}\right\rangle=0.
\end{equation}

Next, the unit vector field $\xi$ tangent to $\mathbb{R}$ in $\mathbb{S}^2\times\mathbb{R}$, viewed as vector field in $\mathbb{R}^4$ has the first three coordinates equal to $0$ and then, since
$$
\xi=\cos\theta\Phi_u+\sin\theta\eta,
$$
one obtains the first three coordinates of $\eta$ as
\begin{equation}\label{eq:etai}
\eta_i=-\cot\theta(\Phi_i)_u,\quad i\in\{1,2,3\}.
\end{equation}

Now we can compute the second order derivatives of $\Phi$. 

First, from Lemma \ref{lemma1} and \eqref{eq:atilde}, we have
\begin{eqnarray*}
\Phi_{uu}&=&\widetilde\nabla_{\frac{\partial}{\partial u}}\frac{\partial}{\partial u}=\bar\nabla_{\frac{\partial}{\partial u}}\frac{\partial}{\partial u}+\widetilde\sigma\left(\frac{\partial}{\partial u},\frac{\partial}{\partial u}\right)\\ &=&-\theta_u\eta-\sin^2\theta\widetilde\xi,
\end{eqnarray*}
which, together with \eqref{eq:etai}, gives
\begin{equation}\label{eq:phiuu}
(\Phi_i)_{uu}=(\theta_u\cot\theta)(\Phi_i)_u-\sin^2\theta\Phi_i,\quad i\in\{1,2,3\}.
\end{equation}

Next, we use again \eqref{eq:atilde}, together with \eqref{eq:beta}, to obtain 
\begin{eqnarray*}
\Phi_{uv}&=&\widetilde\nabla_{\frac{\partial}{\partial u}}\frac{\partial}{\partial v}=\nabla_{\frac{\partial}{\partial u}}\frac{\partial}{\partial v}=\nabla_{X_1}(\beta X_2)\\&=&-\frac{aa_u}{a^2+1}\frac{\partial}{\partial v},
\end{eqnarray*}
which is just
\begin{equation}\label{eq:phiuv}
(\Phi_i)_{uv}=-\frac{aa_u}{a^2+1}(\Phi_i)_v,\quad i\in\{1,2,3\}.
\end{equation}

In the same way, one shows that
$$
\Phi_{vv}=\widetilde\nabla_{\frac{\partial}{\partial v}}\frac{\partial}{\partial v}=\nabla_{\frac{\partial}{\partial v}}\frac{\partial}{\partial v}+a\beta^2\cos\theta\eta-\beta^2\widetilde\xi,
$$
and then
\begin{equation}\label{eq:phivv}
(\Phi_i)_{vv}=\frac{\beta_v}{\beta}(\Phi_i)_v-\frac{a\beta^2}{\sin\theta}(\Phi_i)_u-\beta^2\Phi_i,\quad i\in\{1,2,3\}.
\end{equation}

It is straightforward to solve equation \eqref{eq:phiuv}. We integrate first with respect to $u$ and then to $v$ to obtain 
$$
\Phi_i(u,v)=\frac{\delta_i(v)}{\sqrt{a^2+1}}+\omega_i(u),\quad i\in\{1,2,3\},
$$
where $\delta_i$ and $\omega_i$ are smooth real valued functions. Replacing in equation \eqref{eq:phiuu}, we have 
$$
(\omega_i)_{uu}-\theta_u\cot\theta(\omega_i)_u+\sin^2\theta\omega_i=0,\quad i\in\{1,2,3\}.
$$
Since $\sin\theta\neq 0$ everywhere, we can consider $h_i=(\omega_i)_u/\sin\theta$ and then the above equations become
$$
(h_i^2)_u+(\omega_i^2)_u=0,
$$
i.e., $h_i^2+\omega_i^2=c_{0i}$, where $c_{0i}$, $i\in\{1,2,3\}$, are non-negative real constants. Therefore, we see that
$$
((\omega_i)_u)^2+\frac{(a_u)^2}{(a^2+1)^2}\omega_i^2-c_{0i}\frac{(a_u)^2}{(a^2+1)^2}=0,
$$
with their respective solutions
$$
\omega_i=\frac{c_ia+d_i}{\sqrt{a^2+1}},\quad i\in\{1,2,3\},
$$
where $c_i$ and $d_i$ are real constants. 

Finally, we denote $\varphi_i(v)=\delta_i(v)+d_i$ and then
$$
\Phi_i(u,v)=\frac{\varphi_i(v)+c_ia(u)}{\sqrt{a^2(u)+1}},\quad i=\{1,2,3\}.
$$
Hence, also using \eqref{eq:phi4.1}, we have
\begin{equation}\label{eq:phiinter}
\Phi(u,v)=\left(\frac{1}{\sqrt{a^2(u)+1}}(\varphi(v)+Ca(u)),\int_{u_0}^u\cos\theta(t)dt\right),
\end{equation}
where $\varphi(v)=(\varphi_1(v),\varphi_2(v),\varphi_3(v))$ and $C=(c_1,c_2,c_3)=\cst$ are vectors in $\mathbb{R}^3$. 

The following identities hold in $\mathbb{R}^4$ along the surface
$$
\langle\widetilde\xi,\widetilde\xi\rangle=1,\quad\langle\Phi_u,\Phi_u\rangle=1,\quad\langle\Phi_v,\Phi_v\rangle=\beta^2,\quad\langle\Phi_u,\Phi_v\rangle=0,
$$
and it is straightforward to see that they are equivalent to the following identities in $\mathbb{R}^3$
$$
\langle\varphi,\varphi\rangle=1,\quad\langle C,C\rangle=1,\quad\langle\varphi,C\rangle=0,\quad \langle\varphi_v,\varphi_v\rangle=\tau^2(v).
$$
After changing coordinates, we have $\langle\varphi_v,\varphi_v\rangle=1$. 

Since $C$ is a constant vector, it follows that also $\varphi_v$ is orthogonal to $C$ and then $C$ is collinear with $\varphi\times\varphi_v$. Since $|C|=1$ and $|\varphi\times\varphi_v|=1$, one sees that $\varphi\times\varphi_v=\pm C=\cst$. This implies that $\varphi$ and $\varphi_{vv}$ are collinear. On the other hand, from $\langle\varphi,\varphi_v\rangle=0$ and $\langle\varphi_v,\varphi_v\rangle=1$, one obtains $\langle\varphi,\varphi_{vv}\rangle=-1$. In conclusion, we get that
$$
\varphi_{vv}+\varphi=0,
$$
and then $\varphi(v)=\cos v C_1+\sin v C_2$, where $C_1$ and $C_2$ are two constant vectors in $\mathbb{R}^3$. From $\langle\varphi,\varphi\rangle=\langle\varphi_v,\varphi_v\rangle=1$, one sees that $C_1$ and $C_2$ are orthogonal to each other and of unit length. Replacing in \eqref{eq:phiinter}, we obtain \eqref{eq:local_sxr} and conclude.
\end{proof}

\begin{remark} An integral curve $\gamma_0=\gamma_0(u)$ of $X_1=\partial/\partial u$ is given by
$$
\gamma_0(u)=\left(\frac{\cos v_0}{\sqrt{a^2+1}}C_1+\frac{\sin v_0}{\sqrt{a^2+1}}C_2+\frac{a}{\sqrt{a^2+1}}(C_1\times C_2),\int_{u_0}^u\cos\theta(t)dt\right),
$$
for some fixed $v_0$. After a change of coordinates, $\gamma_0$ can be written as
$$
\gamma_0(u)=\left(\frac{1}{\sqrt{a^2+1}},0,\frac{a}{\sqrt{a^2+1}},\int_{u_0}^u\cos\theta(t)dt\right),
$$
which shows that the curve lies on the cylinder $\mathbb{S}^1\times\mathbb{R}$. Moreover, since $\gamma_0$ is a geodesic of $\Sigma^2$, we have $\bar\nabla_{\gamma_0'}\gamma'_0=\kappa N$, where $\kappa=|\theta'|$ is its curvature and $N=\mp\eta$, as $\theta'$ is positive or negative, and then $\bar\nabla_{\gamma_0'}N=\pm\bar\nabla_{\gamma_0'}\eta=\mp\theta'\gamma_0'=-\kappa\gamma_0'$, which shows that $\gamma_0$ is a plane curve. Hence, an integral curve of $X_1=\partial/\partial u$ is an (open part of an) ellipse on the cylinder $\mathbb{S}^1\times\mathbb{R}$.
\end{remark}

\begin{remark} We note that equations \eqref{eq:phivv} are also satisfied by surfaces given by Theorem \ref{thm:local_sxr}.
\end{remark}

\subsection{Biconservative surfaces in $\mathbb{H}^2\times\mathbb{R}$} Let $\mathbb{L}^3$ be the three-dimensional Lorentz-Minkowski space, i.e., the real vector space $\mathbb{R}^3$ with the Lorentzian metric tensor $\langle,\rangle_{\mathbb{L}^3}$ given by
$$
\langle,\rangle_{\mathbb{L}^3}=dx_1^2+dx_2^2-dx_3^2,
$$
where $(x_1,x_2,x_3)$ are the canonical coordinates of $\mathbb{R}^3$. We will use the following model of the hyperbolic plane $\mathbb{H}^2$, viewed as a quadric in $\mathbb{L}^3$,
$$
\mathbb{H}^2=\{(x_1,x_2,x_3)\in\mathbb{L}^3:x_1^2+x_2^2-x_3^2=-1,\quad x_3>0\},
$$
with the induced metric, which is a Riemannian one with constant sectional curvature equal to $-1$. 

We also recall the definition of the Lorentzian cross-product $\otimes:\mathbb{R}^3\times\mathbb{R}^3\rightarrow\mathbb{R}^3$ (see, for example, \cite{R} for more details)
$$
(x_1,x_2,x_3)\otimes(y_1,y_2,y_3)=(x_2y_3-x_3y_2,x_3y_1-x_1y_3,x_2y_1-x_1y_2),
$$
with the following properties, similar to those of the vector cross product in Euclidean spaces,
$$
\langle X\otimes Y,X\rangle_{\mathbb{L}^3}=\langle X\otimes Y,Y\rangle_{\mathbb{L}^3}=0,\quad
X\otimes Y=-Y\otimes X,
$$
and
$$
\langle X\otimes Y,X\otimes Y\rangle_{\mathbb{L}^3}=-\langle X,X\rangle_{\mathbb{L}^3}\langle Y,Y\rangle_{\mathbb{L}^3}+\langle X,Y\rangle_{\mathbb{L}^3}^2,
$$
where $X$ and $Y$ are vectors in $\mathbb{R}^3$.

The following result gives the local equation of a non-CMC biconservative surface with $\nabla f$ a principal direction in $\mathbb{H}^2\times\mathbb{R}$. We will see that we have three possible cases, according to the values of the function $a$ introduced in Proposition \ref{prop1}. Thus, either $a^2=1$ throughout the open set $U$, or, possibly restricting to a smaller open subset $W\subset U$, $a^2\neq 1$ at each point. In the later situation there actually are two different cases, as $a^2>1$ or $a^2<1$.

\begin{theorem}\label{thm:local_hxr} Let $\Sigma^2$ be a biconservative surface in $\mathbb{H}^2\times\mathbb{R}$ with $f>0$ and $\nabla f\neq 0$ at each point on the surface, such that $\nabla f$ is a principal direction. Then $\Sigma^2\subset\mathbb{L}^3\times\mathbb{R}$ can be locally parametrized as follows
\begin{enumerate}

\item[(i)] if $a=\pm 1$
\begin{equation*}\label{eq:local_hxr1}
\Phi(u,v)=\left(e^{\pm\int_{u_0}^u\sin\theta(t)dt}\left(C_0+C_1v+C_2v^2\right)+C_2e^{\mp\int_{u_0}^u\sin\theta(t)dt},\int_{u_0}^u\cos\theta(t)dt\right),
\end{equation*}
where $C_0,C_1,C_2\in\mathbb{R}^3$ are constant vectors such that
$$
\langle C_0,C_0\rangle_{\mathbb{L}^3}=\langle C_2,C_2\rangle_{\mathbb{L}^3}=0,\quad \langle C_1,C_1\rangle_{\mathbb{L}^3}=1,
$$
$$
\langle C_0,C_1\rangle_{\mathbb{L}^3}=\langle C_1,C_2\rangle_{\mathbb{L}^3}=0,\quad \langle C_0,C_2\rangle_{\mathbb{L}^3}=-\frac{1}{2},
$$
and $\theta=\theta(u)$ is a real valued function satisfying 
\begin{equation}\label{eq:theta1}
\ln(\cos^2\theta)+\ln\left(\sqrt{g^2(6g^2\mp 5g+2)}\right)\mp\frac{5\sqrt{23}}{23}\arctan\left(\frac{12g-5}{\sqrt{23}}\right)+c=0,
\end{equation}
with $\cos\theta\neq 0$ and $\sin\theta\neq 0$ at each point on the surface, where $g=g(u)$ is a function given by $\theta'=(\pm 1-2g)\cos\theta$ and $g\cos\theta>0$ everywhere, and $c\in\mathbb{R}$ is a constant.

\item[(ii)] 

\begin{enumerate}

\item if $a^2>1$
\begin{equation*}\label{eq:local_hxr2}
\Phi(u,v)=\left(\frac{\cos v}{\sqrt{a^2-1}}C_1+\frac{\sin v}{\sqrt{a^2-1}}C_2+\frac{a}{\sqrt{a^2-1}}(C_1\otimes C_2),\int_{u_0}^u\cos\theta(t)dt\right),
\end{equation*}
where $C_1,C_2\in\mathbb{R}^3$ are two constant vectors such that
$$
\langle C_1,C_1\rangle_{\mathbb{L}^3}=\langle C_2,C_2\rangle_{\mathbb{L}^3}=1,\quad \langle C_1,C_2\rangle_{\mathbb{L}^3}=0, 
$$

\item if $a^2<1$
\begin{equation*}\label{eq:local_hxr3}
\Phi(u,v)=\left(\frac{e^v}{\sqrt{1-a^2}}C_1+\frac{e^{-v}}{\sqrt{1-a^2}}C_2+\frac{2a}{\sqrt{1-a^2}}(C_1\otimes C_2),\int_{u_0}^u\cos\theta(t)dt\right),
\end{equation*}
where $C_1,C_2\in\mathbb{R}^3$ are two constant vectors such that
$$
\langle C_1,C_1\rangle_{\mathbb{L}^3}=\langle C_2,C_2\rangle_{\mathbb{L}^3}=0,\quad \langle C_1,C_2\rangle_{\mathbb{L}^3}=-\frac{1}{2},
$$
\end{enumerate}
\noindent with functions $a=a(u)$ and $\theta=\theta(u)$ satisfying 
$$
\sin\theta\neq 0,\quad \cos\theta\neq 0,\quad a'+(a^2-1)\sin\theta=0,
$$ 
and
\begin{equation}\label{eq:bicons_hxr3}
\left(\theta'-a\cos\theta\right)\left(\theta'-a\cos\theta\right)'+(a^2-1)\left(\frac{(\theta')^2-a^2\cos^2\theta}{a^2-1}\right)'=0.
\end{equation}

\end{enumerate}
\end{theorem}

\begin{proof} In all three cases of the theorem, the first part of the proof is similar to that of Theorem \ref{thm:local_sxr} and that is why we will only point out the main steps and formulas that lead to the local equations. 

First, since by Lemma \ref{lemma1} and Proposition \ref{prop1} we have $[X_1,X_2]=-\nabla_{X_2}X_1=-\lambda_2\tan\theta X_2=-a\sin\theta X_2$, it follows that we can choose local coordinates $\{u,v\}$ on the surface such that $\partial/\partial u=X_1$ and $\partial/\partial v=\beta(u,v)X_2$, where $\beta$ is a real valued function given by
\begin{equation}\label{eq:beta_hxr}
\frac{\partial\beta}{\partial u}=a\sin\theta\beta,
\end{equation}
where $a=a(u)$ and $\theta=\theta(u)$, since $X_2(a)=X_2(\theta)=0$. 

The local equation of the surface $\Sigma^2$ in $\mathbb{L}^3\times\mathbb{R}$ is 
$$
\Phi=\Phi(u,v)=(\Phi_1(u,v),\Phi_2(u,v),\Phi_3(u,v),\Phi_4(u,v)),
$$
with
\begin{equation}\label{eq:phi4.2}
\Phi_4(u,v)=\Phi_4(u)=\int_{u_0}^u\cos\theta(t)dt.
\end{equation}

Let $\widetilde\xi$ be a unit vector field in $\mathbb{L}^3\times\mathbb{R}$ normal to $\mathbb{H}^2\times\mathbb{R}$. Along $\Sigma^2$ this vector field is given by $\widetilde\xi=(\Phi_1,\Phi_2,\Phi_3,0)$ and $\langle\widetilde\xi,\widetilde\xi\rangle=-1$. Then, we have
\begin{equation}\label{eq:atilde2}
\left\langle\widetilde A_{\widetilde\xi}\frac{\partial}{\partial u},\frac{\partial}{\partial u}\right\rangle=-\sin^2\theta,\quad \left\langle\widetilde A_{\widetilde\xi}\frac{\partial}{\partial v},\frac{\partial}{\partial v}\right\rangle=-\beta^2,\quad \left\langle\widetilde A_{\widetilde\xi}\frac{\partial}{\partial u},\frac{\partial}{\partial v}\right\rangle=0.
\end{equation}
Considering all these, in the same way as for surfaces in $\mathbb{S}^2\times\mathbb{R}$, one obtains the following equations
\begin{equation}\label{eq:phiuu2}
(\Phi_i)_{uu}=(\theta_u\cot\theta)(\Phi_i)_u+\sin^2\theta\Phi_i,\quad i\in\{1,2,3\},
\end{equation}
\begin{equation}\label{eq:phiuv2}
(\Phi_i)_{uv}=a\sin\theta(\Phi_i)_v,\quad i\in\{1,2,3\},
\end{equation}
and
\begin{equation}\label{eq:phivv2}
(\Phi_i)_{vv}=\frac{\beta_v}{\beta}(\Phi_i)_v-\frac{a\beta^2}{\sin\theta}(\Phi_i)_u+\beta^2\Phi_i,\quad i\in\{1,2,3\}.
\end{equation}

Let us now split our study according to the possible values of the function $a$. 

If $a=\pm 1$ throughout the surface, from \eqref{eq:beta_hxr}, one obtains $\beta=\beta(u,v)=\tau(v)e^{\pm\int_{u_0}^u\sin\theta(t)dt}$ and, solving equations \eqref{eq:phiuv2} and \eqref{eq:phiuu2}, we have that
$$
\Phi(u,v)=\left(\varphi(v)e^{\pm\int_{u_0}^u\sin\theta(t)dt}+Ce^{\mp\int_{u_0}^u\sin\theta(t)dt},\int_{u_0}^u\cos\theta(t)dt\right),
$$
where $\varphi(v)=(\varphi_1(v),\varphi_2(v),\varphi_3(v))$ and $C=(c_1,c_2,c_3)=\cst$ are vectors in $\mathbb{R}^3$, such that
$$
\Phi_1^2+\Phi_2^2-\Phi_3^2=-1,\quad\langle\Phi_u,\Phi_u\rangle=1,\quad\langle\Phi_u,\Phi_v\rangle=0,\quad\langle\Phi_v,\Phi_v\rangle=\beta^2,
$$ 
or, equivalently,
\begin{equation}\label{eq:vector1}
\langle\varphi(v),\varphi(v)\rangle_{\mathbb{L}^3}=\langle C,C\rangle_{\mathbb{L}^3}=0,\quad\langle \varphi(v),C\rangle_{\mathbb{L}^3}=-\frac{1}{2},\quad\langle \varphi'(v),\varphi'(v)\rangle_{\mathbb{L}^3}=\tau^2(v).
\end{equation}
We change the coordinates such that the last identity becomes $\langle \varphi'(v),\varphi'(v)\rangle_{\mathbb{L}^3}=1$. 

Next, for a fixed $u_1$, consider an integral curve $\gamma=\gamma(s)=\gamma(u_1,s)$ of $X_2$, parametrized by arc length. From Lemma~\ref{lemma1}, Proposition \ref{prop1}, and equation \eqref{eq:atilde2}, we have
\begin{eqnarray*}
\gamma''&=&\widetilde\nabla_{X_2}X_2=\bar\nabla_{X_2}X_2+\widetilde\sigma(X_2,X_2)\\
&=&\mp\sin\theta X_1\pm\cos\theta\eta+\widetilde\xi,
\end{eqnarray*}
and then
$$
\gamma'''=\widetilde\nabla_{X_2}\widetilde\nabla_{X_2}X_2=0.
$$
Then, $\gamma(s)=\gamma(u_1,s)=\left(C_0+C_1s+C_2s^2,\int_{u_0}^{u_1}\cos\theta(t)dt\right)$, where $C_0$, $C_1$, and $C_2$ are constant vectors in $\mathbb{R}^3$. But, for $u_1$, the integral curve of $X_2$ is $\Phi(u_1,v)$, that can be parametrized by arc length as
$$
\Phi(u_1,s)=\left(\varphi\left(\frac{s}{\alpha}\right)e^{\pm\int_{u_0}^{u_1}\sin\theta(t)dt}+Ce^{\mp\int_{u_0}^{u_1}\sin\theta(t)dt},\int_{u_0}^{u_1}\cos\theta(t)dt\right),
$$
by taking $v=s/\alpha$, where $\alpha=e^{\pm\int_{u_0}^{u_1}\sin\theta(t)dt}$. Thus, we obtain $\varphi(v)=C_0+C_1v+C_2v^2$. From equations \eqref{eq:vector1}, one easily shows that
$$
\langle C_0,C_0\rangle_{\mathbb{L}^3}=\langle C_0,C_1\rangle_{\mathbb{L}^3}=\langle C_1,C_2\rangle_{\mathbb{L}^3}=\langle C_2,C_2\rangle_{\mathbb{L}^3}=0, 
$$
$$
\langle C_0,C_2\rangle_{\mathbb{L}^3}=-\frac{1}{2},\quad \langle C_1,C_1\rangle_{\mathbb{L}^3}=1,
$$
and
$$
\langle C_0,C\rangle_{\mathbb{L}^3}=-\frac{1}{2},\quad \langle C_1,C\rangle_{\mathbb{L}^3}=\langle C_2,C\rangle_{\mathbb{L}^3}=0.
$$
The vector $C_1$ is not light-like and therefore the other three vectors $C_0$, $C_2$, and $C$ lie in the same plane. It is easy then to see, from the above equations, that $C_2=C$.

To conclude the proof of the first item of the theorem, let us recall the biconservative equation \eqref{eq:l2}, which, in this case, assumes the form
$$
\pm\cos\theta=3f+\frac{f}{f'}\sin\theta\cos\theta,
$$
where $f$ is the mean curvature function and $\theta'=\pm\cos\theta-2f$. We shall consider only the first case, when $a=1$, i.e., $\lambda_2=\cos\theta$, the treatment of the other one being quite similar. Let us look for $f$ of the form $f=g\cos\theta$. It is straightforward to show that the biconservative equation becomes
$$
(1-3g)(2g-1)g'=(6g^3-5g^2+2g)(\ln(|\cos\theta|)',
$$
whose solution is given by \eqref{eq:theta1}.

Next, assume that $a^2\neq 1$ throughout the surface, which means that $a^2-1$ does not change the sign on $\Sigma^2$. In this case, equation \eqref{eq:beta_hxr} and Proposition \ref{prop1} give $\beta(u,v)=\tau(v)/\sqrt{|a^2-1|}$, and \eqref{eq:phiuu2} and \eqref{eq:phiuv2} lead to
$$
\Phi(u,v)=\left(\frac{1}{\sqrt{|a^2-1|}}(\varphi(v)+Ca(u)),\int_{u_0}^{u}\cos\theta(t)dt\right),
$$
where $\varphi(v)$ and $C=\cst$ are vectors in $\mathbb{R}^3$. 

As in the previous case, we must have 
$$
\Phi_1^2+\Phi_2^2-\Phi_3^2=-1,\quad\langle\Phi_u,\Phi_u\rangle=1,\quad\langle\Phi_u,\Phi_v\rangle=0,\quad\langle\Phi_v,\Phi_v\rangle=\beta^2,
$$ 
and, therefore,
\begin{equation}\label{eq:vector2}
\langle\varphi(v),\varphi(v)\rangle_{\mathbb{L}^3}=\pm1,\quad\langle C,C\rangle_{\mathbb{L}^3}=\mp 1,\quad\langle \varphi(v),C\rangle_{\mathbb{L}^3}=0,\quad\langle \varphi'(v),\varphi'(v)\rangle_{\mathbb{L}^3}=\tau^2(v),
\end{equation}
as $a^2>1$ or $a^2<1$. We change the coordinates in this situation too, to get $\langle \varphi'(v),\varphi'(v)\rangle_{\mathbb{L}^3}=1$.

Now, let us consider an integral curve $\gamma=\gamma(s)=\gamma(u_1,s)$ of $X_2$, parametrized by arc length. As before, using Lemma~\ref{lemma1}, Proposition \ref{prop1}, and equation \eqref{eq:atilde2}, we have
\begin{eqnarray*}
\gamma''&=&\widetilde\nabla_{X_2}X_2=\bar\nabla_{X_2}X_2+\widetilde\sigma(X_2,X_2)\\
&=&-a\sin\theta X_1+a\cos\theta\eta+\widetilde\xi,
\end{eqnarray*}
and then
$$
\gamma'''=\widetilde\nabla_{X_2}\widetilde\nabla_{X_2}X_2=-(a^2-1)X_2.
$$
Hence, the curve $\gamma$ satisfies the equation
\begin{equation}\label{eq:gamma}
\gamma'''+(a^2-1)\gamma'=0,
\end{equation}
in $\mathbb{L}^3\times\mathbb{R}$.

If $a^2>1$, working as in the previous case, this equation shows that
$$
\varphi(v)=C_0+C_1\cos v+C_2\sin v,
$$
with $C_0$, $C_1$, and $C_2$ constant vectors in $\mathbb{R}^3$. From equations \eqref{eq:vector2}, one obtains
$$
\langle C_0,C_0\rangle_{\mathbb{L}^3}=0,\quad\langle C_0,C_1\rangle_{\mathbb{L}^3}=\langle C_0,C_2\rangle_{\mathbb{L}^3}=\langle C_1,C_2\rangle_{\mathbb{L}^3}=0, 
$$
$$
\langle C_1,C_1\rangle_{\mathbb{L}^3}=\langle C_2,C_2\rangle_{\mathbb{L}^3}=1,
$$
and
$$
\langle C_0,C\rangle_{\mathbb{L}^3}=\langle C_1,C\rangle_{\mathbb{L}^3}=\langle C_2,C\rangle_{\mathbb{L}^3}=0.
$$
Using the properties of the Lorentzian cross-product, since the vector $C$ is not like-like, it follows that it is orthogonal to $C_1$ and $C_2$, and, therefore, collinear with $C_1\otimes C_2$.  We also have $\langle C_1\otimes C_2,C_1\otimes C_2\rangle_{\mathbb{L}^3}=-1$ and this shows that we can choose $C=C_1\otimes C_2$.  Moreover, as none of the vectors $C_1$, $C_2$, and $C_1\otimes C_2$ is light-like, it follows, from the above identities, that $C_0$ must vanish.

Finally, if $a^2<1$, equation \eqref{eq:gamma} leads to
$$
\varphi(v)=C_0+C_1e^v+C_2e^{-v},
$$
where $C_0$, $C_1$, and $C_2$ are constant vectors in $\mathbb{R}^3$. Again using \eqref{eq:vector2}, we have
$$
\langle C_0,C_0\rangle_{\mathbb{L}^3}=\langle C_0,C_1\rangle_{\mathbb{L}^3}=\langle C_0,C_2\rangle_{\mathbb{L}^3}=0,\quad\langle C_1,C_2\rangle_{\mathbb{L}^3}=-\frac{1}{2}, 
$$
$$
\langle C_1,C_1\rangle_{\mathbb{L}^3}=\langle C_2,C_2\rangle_{\mathbb{L}^3}=0,
$$
and
$$
\langle C_0,C\rangle_{\mathbb{L}^3}=\langle C_1,C\rangle_{\mathbb{L}^3}=\langle C_2,C\rangle_{\mathbb{L}^3}=0.
$$
We also have $\langle C_1\otimes C_2,C_1\otimes C_2\rangle_{\mathbb{L}^3}=1/4$ and, since $C$ is orthogonal to $C_1$ and $C_2$, one can choose $C=2(C_1\otimes C_2)$.  In the same way as above, it follows that $C_0$ vanishes.

As in the case of biconservative surfaces in $\mathbb{S}^2\times\mathbb{R}$, it is straightforward to put the biconservative condition \eqref{eq:l1} in the form \eqref{eq:bicons_hxr3}, which concludes the proof.
\end{proof}

\begin{remark} For biconservative surfaces in $\mathbb{H}^2\times\mathbb{R}$ given by Theorem \ref{thm:local_hxr}, one observes that identities \eqref{eq:phivv2} are satisfied, even though they were not explicitly used to determine the local equations of the surfaces.
\end{remark}

\begin{remark} The Gaussian curvature of a non-CMC biconservative surface $\Sigma^2$ in $M^2(c)\times\mathbb{R}$, with $\nabla f$ a principal direction, is given by 
$$
K=K(u)=\lambda_1\lambda_2+c\sin^2\theta=2af\cos\theta-a^2\cos^2\theta+c\sin^2\theta.
$$
If we denote $f=g\cos\theta$, a straightforward computation leads to
$$
K'=2a(a^2+c)g'+(10a^2+6c)a'g-8aa'g^2-4aa'(a^2+c),
$$
if either $c=1$ or $c=-1$ and $a^2\neq 1$ at a point $q\in U$. In the later case we work on an open neighborhood $W\subset U$ of $q$ such that $a^2\neq 1$ everywhere on $W$. 

Let us assume that $a^2+c$ does not vanish identically on $U$. The surface has constant Gaussian curvature if and only if $K'=0$, and this condition can be written, in this case, as
\begin{equation}\label{eq:11}
2a(a^2+c)\frac{dg}{da}=8ag^2-(10a^2+6c)g+4a(a^2+c).
\end{equation}

In general, the biconservative equation \eqref{eq:l2} can be written as
\begin{equation}\label{eq:12}
(a^2+c)(3g-a)\frac{dg}{da}=6g^3-5ag^2+(a^2+c)g.
\end{equation}

If $c=1$, then \eqref{eq:11} is a Riccati equation with the general solution 
$$
g=\frac{a}{2}+\frac{a(a^2+1)}{c_0a^4+1},\quad c_0=\cst>0.
$$
It is then easy to verify that $g$ does not satisfy \eqref{eq:12}.

If $c=-1$ the general solution of equation \eqref{eq:11} is
$$
g=\frac{a}{2}+\frac{a^3}{c_0a^4-2a^2(c_0-1)+c_0-1},\quad c_0=\cst>0,
$$
which does not satisfy the biconservative equation \eqref{eq:12}.

Finally, when $c=-1$ and $a=\pm 1$, the Gaussian curvature is constant if and only if
$$
g'=(4g\mp 8g^2)\sin\theta,
$$
and it is again straightforward to verify that such a surface cannot be biconservative in this case too.

Hence, there are no non-CMC biconservative surfaces with $\nabla f$ a principal direction and constant Gaussian curvature in $M^2(c)\times\mathbb{R}$. 
\end{remark}


\begin{thebibliography}{99}


\bibitem{AK} \c S. Andronic, A. Kayhan, {\it Rigidity results for compact biconservative hypersurfaces in space forms}, J. Geom. Phys. 212 (2025), Paper No. 105460.


\bibitem{CMOP} R. Caddeo, S. Montaldo, C. Oniciuc, and P. Piu, {\it Surfaces in three-dimensional space forms with divergence-free stress-bienergy tensor}, Ann. Mat. Pura Appl. (4) 193 (2014), 529--550.

\bibitem{43} B. Y. Chen, {\it Some open problems and conjectures on submanifolds of finite type}, Soochow J. Math. 17 (1991), 169--188.

\bibitem{Ishikawa} B. Y. Chen, S. Ishikawa, {\it Biharmonic surfaces in pseudo-Euclidean spaces}, Mem. Fac. Sci. Kyushu Univ. Ser. A 45 (1991), 323--347.




\bibitem{ES} J.~Eells, J. H.~Sampson, \textit{Harmonic mappings of Riemannian manifolds}, Amer. J. Math. 86 (1964), 109--160.


\bibitem{FO} D. Fetcu, C. Oniciuc, {\it Biharmonic and biconservative hypersurfaces in space forms}, Differential geometry and global analysis--in honor of Tadashi Nagano, 65--90, Contemp. Math. 777, Amer. Math. Soc., [Providence], RI, 2022.


\bibitem{FSol} D. Fetcu, {\it The rigidity of biconservative surfaces in $\sol$}, Ann. Mat Pura Appl., to appear.


\bibitem{Hasanis} T. Hasanis, T. Vlachos, {\it Hypersurfaces in $\mathbb{R}^{4}$ with harmonic mean curvature vector field}, Math. Nachr. 172 (1995), 145--169.

\bibitem{Jiang} G. Y. Jiang, {\it 2-harmonic maps and their first and second variational formulas}, Chinese Ann. Math. Ser. A 7 (1986), 389--402.

\bibitem{Jiang87} G. Y. Jiang, {\it The conservation law for 2-harmonic maps between Riemannian manifolds}, Acta Math. Sinica 30 (1987),  220--225.

\bibitem{J2} G. Y. Jiang, {\it Some nonexistence theorems on 2-harmonic and isometric immersions in Euclidean space}, Chinese Ann. Math. Ser. B 8 (1987), 377--383.



\bibitem{LMO} E. Loubeau, S., Montaldo, and C. Oniciuc, {\it The stress-energy tensor for biharmonic maps}, Math. Z. 259 (2008), 503--524.





\bibitem{MOR} S. Montaldo, C. Oniciuc, and A. Ratto, {\it Biconservative surfaces}, J. Geom. Anal. 26 (2016), 313--329.

\bibitem{MOP} S. Montaldo, I. I. Onnis, and A. P. Passamani, {\it Biconservative surfaces in BCV-spaces}, Math. Nachr. 290 (2017), 2661--2672.


\bibitem{MP} S. Montaldo, A. P\'ampano, {\it On the existence of closed biconservative surfaces in space forms}, Comm. Anal. Geom. 31 (2023), 291--319. 

\bibitem{SN} S. Nistor, {\it Complete biconservative surfaces in $\mathbb{R}^3$ and $\mathbb{S}^3$}, J. Geom. Phys. 110 (2016), 130--153.

\bibitem{Nistor1} S. Nistor, {\it On biconservative surfaces}, Differential Geom. Appl. 54 (2017), 490--502.

\bibitem{NO} S. Nistor, C. Oniciuc,  {\it On the uniqueness of complete biconservative surfaces in 3-dimensional space forms}, Ann. Sc. Norm. Super. Pisa Cl. Sci. (5) 23 (2022), 1565--1587.

\bibitem{Ou} Y. L. Ou,, {\it Biharmonic hypersurfaces in Riemannian manifolds}, Pacific J. Math. 248 (2010), 217--232.

\bibitem{Chen-Ou} Y. L. Ou, B. Y. Chen, {\it Biharmonic Submanifolds and Biharmonic Maps in Riemannian Geometry}, World Scientific Publishing Co. Pte. Ltd., Hackensack, NJ, 2020.

\bibitem{OW} Y. L. Ou, Z. P. Wang, {\it Constant mean curvature and totally umbilical biharmonic surfaces in 3-dimensional geometries}, J. Geom. Phys. 61 (2011), 1845--1853.


\bibitem{R} J. G. Ratcliffe, {\it Foundations of Hyperbolic Manifolds}, Third edition. Graduate Texts in Mathematics, 149, Springer, Cham, 2019. 

\end{thebibliography}
\end{document}